\documentclass[12pt,a4paper,french,english]{amsart}

\usepackage{amsfonts,amsmath,amssymb, amscd,fullpage}
\usepackage{stmaryrd}
\usepackage[all]{xy}
\usepackage[T1]{fontenc}
\usepackage{lmodern}
 \usepackage[utf8]{inputenc}
\usepackage[french, english]{babel}

\newtheorem{theorem}{Theorem}[section]
\newtheorem{lemma}[theorem]{Lemma}
\newtheorem{proposition}[theorem]{Proposition}
\newtheorem{corollary}[theorem]{Corollary}

\theoremstyle{definition}
\newtheorem{definition}[theorem]{Definition}

\theoremstyle{remark}
\newtheorem{remark}[theorem]{Remark}

\newtheorem{example}[theorem]{Example}

\newcommand\pf{\begin{proof}}
\newcommand\epf{\end{proof}}

\newcommand\C{\mathbb{C}}

\newcommand\coc{\mathcal C}
\newcommand\yd{\mathcal{YD}}

\newcommand\cd{\mathrm{cd}}
\newcommand\pd{\mathrm{pd}}

\DeclareMathOperator{\Hom}{Hom}

\numberwithin{equation}{section}

\title{Twisted separability for adjoint functors}

\author{Julien Bichon}
\address{ Universit\'e Clermont Auvergne, CNRS, LMBP, F-63000 CLERMONT-FERRAND, FRANCE}
\email{julien.bichon@uca.fr}


\begin{document}

\begin{abstract}
Twisted separable functors generalize the separable functors of Nastasescu, Van den Bergh and Van Oystaeyen, and provide a convenient tool to compare various projective dimensions. We discuss when an adjoint functor is twisted separable, obtaining a version of Rafael's Theorem in the twisted case. As an application, we show that if $R$ is Hopf-Galois object over a Hopf algebra $A$, then their Hochschild cohomological dimension coincide, provided that the cohomological dimension of $A$ is finite and that $R$ has a unital twisted trace with respect to a semi-colinear automorphism.
\end{abstract}

\maketitle




\section{Introduction}

Separable functors, introduced by Nastasescu, Van den Bergh and Van Oystaeyen \cite{nvdbvo},  provide an elegant and powerful categorical setting for proving various types of generalized Maschke theorems. We refer to  \cite{camizh} for a systematic presentation and many examples in the setting of generalized Hopf modules categories.

One nice feature of separable functors is that they reasonably preserve projective dimensions whenever these are defined: if $F  : \mathcal{C}\to \mathcal{D}$ is an exact functor between abelian categories having enough projective objects, then if $F$ preserve projectives and is separable, we have $\pd_{\mathcal{C}}(V) = \pd_{\mathcal{D}}(F(V))$ for any object $V$ in $\mathcal{C}$.

Motivated by the question of the monoidal invariance of the cohomological dimension of Hopf algebras, the notion of twisted separable functor was introduced in \cite{bi22}. Twisted separable functors provide a flexible generalization of separable functors, having a similar property of preservation of projective dimensions when these are finite. They were used to show the following key result in \cite{bi22} (Theorem 23 there): 
\textsl{if $A$ is a  Hopf  algebra and and $R$ is a Hopf-Galois object over $A$, then if $A$ is cosemisimple and $\cd(A)$ is finite, we have $\cd(A)=\cd(R)$.} 
Here $\cd(R)$ denotes the Hochschild cohomological dimension, i.e. the projective dimension of $R$ as an $R$-bimodule. 

In this paper we develop some more theory on twisted separable functors. We are particularly interested in Rafael's theorem, a very convenient characterization of separability for functors that are adjoint functors, given in \cite[Theorem 1.2]{raf}. We provide a sufficient condition on a left adjoint functor to be twisted separable in the spirit of \cite{raf}, see Theorem \ref{thm:raftw}, which enables us to replace the cosemisimplicity condition on $A$ in the above result by a  condition on $R$, as follows.

\begin{theorem} \label{thm:cdhgtt}
Let $A$ be Hopf algebra with bijective antipode and let $R$ be a Hopf-Galois object over $A$. If $R$ has a unital twisted trace with respect to a semi-$A$-colinear automorphism of $R$ and $\cd(A)$ is finite, we have $\cd(A)=\cd(R)$. 
\end{theorem}

The notion of  twisted trace in the statement is the following one: a twisted trace on $R$ is a linear map $\varphi : R \to k$ for which there exists an algebra automorphism $\sigma$ of $R$ such that $\varphi(xy) = \varphi(y\sigma(x))$ for any $x,y \in R$. The twisted trace is said to be unital if $\varphi(1)=1$. We refer to Section \ref{sec:apphg} for the notion of semi-colinearity for a twisted trace, but notice that a trace in the usual sense (i.e. with $\sigma$ above being the identity map) is always semi-colinear.  

The strategy to prove Theorem \ref{thm:cdhgtt} is as follows. If $A$ is a Hopf algebra and $R$ is a left $A$-comodule algebra, there is always an associated ``canonical'' functor  $F_R : {_A \mathcal M} \longrightarrow {_R\mathcal M}_R$ sending the trivial $A$-module $k$ to the $R$-bimodule $R$. Since the cohomological dimension of $A$ equals $\pd_A(k)$, the projective dimension of $k$ as an $A$-module  (see \cite{giku} for example), the comparison of $\cd(A)$ and $\cd(R)$ boils down to the question of the preservation of projective dimensions by $F_R$. The assumptions in Theorem \ref{thm:cdhgtt} then ensure the twisted separability of $F_R$ for $R$ Hopf-Galois (see Theorem  \ref{thm:tshg}), and hence the statement will follow from the preservation of finite projective dimensions by a twisted separable functor.

The paper is organized as follows. 
Section \ref{sec:tsf} recalls the notion of twisted separable functor. In Section \ref{sec:tsfa} we give the characterization of twisted separability for adjoint functors. Section \ref{sec:apphg} is devoted to study the twisted separability of the canonical functor  $F_R : {_A \mathcal M} \longrightarrow {_R\mathcal M}_R$ associated to a left Galois object $R$ over a Hopf algebra $A$, which then is used to prove Theorem \ref{thm:cdhgtt}. In the final Section \ref{sec:tsfmod} we discuss twisted separability in the classical situation of functors between categories of modules.

\medskip

\noindent
\textbf{Notations and conventions.}
We work over a fixed field $k$. The category of left modules over an algebra $R$ is denoted ${_R \mathcal M}$, the category of right $R$-modules is denoted $\mathcal M_R$, the category of $R$-bimodules is denoted  ${_R\mathcal M}_R $.
We assume that the reader is familiar with the theory of Hopf algebras and their modules and comodules, as e.g. in \cite{mon}, and with the basics of homological algebra \cite{wei}, and most notably projective dimensions.
If $A$ is a Hopf algebra, as usual, $\Delta$, $\varepsilon$ and $S$ stand respectively for the comultiplication, counit and antipode of $A$. We use Sweedler's notations in the standard way. 
The trivial $A$-module corresponding to the counit of $A$ is denoted $k$.

\section{Twisted separable functors}\label{sec:tsf}

In this section we recall the notion of twisted separable functor from \cite{bi22}, that we supplement with some new additional vocabulary.

\subsection{Definition and basic examples}

If $\mathcal C$ is category, we say that a subclass $\mathcal{F}$ of objects of $\mathcal{C}$ is generating if for every object $V$ of $\mathcal{C}$, there exists an object $P$ of $\mathcal{F}$ together with an epimorphism $P \to V$.

\begin{definition}\label{def:tsf}
Let $\mathcal C$ and $\mathcal D$ be some categories. A functor $F : \mathcal{C} \to \mathcal{D}$ is said to be \textsl{twisted separable} if there exist
\begin{enumerate}
	\item an autoequivalence $\Theta$ of the category $\mathcal{D}$;
	\item a generating subclass $\mathcal{F}$ of objects of $\mathcal{C}$  together with,  for any object $P$ of $\mathcal{F}$, an isomorphism $\theta_P : F(P)\to \Theta F(P)$;
	\item a natural morphism $\textbf{M}_{-,-} : \Hom_{\mathcal{D}}(F(-),\Theta F(-)) \to \Hom_{\mathcal{C}}(-,-)$ such that for any object $P$ of $\mathcal{F}$, we have $\textbf{M}_{P,P}(\theta_P)= {\rm id}_P$.
\end{enumerate}
A triple $(\Theta, \mathcal F, \theta_{-})$ as above is said to be a \textsl{twisted separability datum} for $F$.
\end{definition}

The naturality condition above means that for any morphisms $\alpha : V'\to V$, $\beta : W \to W'$ in $\mathcal{C}$ and any morphism $f : F(V) \to \Theta F(W)$ in $\mathcal{D}$, we have 
$$\beta \circ {\bf M}_{V,W}(f) \circ \alpha = {\bf M}_{V',W'}(\Theta F(\beta) \circ f \circ F(\alpha))$$

A twisted separable functor  $F : \mathcal{C} \to \mathcal{D}$ with twisted separability datum $({\rm id}_\mathcal{C}, {\rm ob}(\mathcal C), {\rm id})$ is a separable functor in the sense of 
\cite{nvdbvo}. 
\begin{example}
 Let $G$ a finite group and let $\Omega : {\rm Rep}(G) \to  {\rm Vec}_k$ be the forgetful functor from the category of $k$-linear representations of $G$ to the category of vector spaces. A version of the classical Maschke theorem is that $\Omega$ is separable if and only if $|G|\not = 0$ in $k$. 
\end{example}

The following example from \cite{bi22} shows that twisted separability is indeed less restrictive than separability.

\begin{example} \label{ex:ydsep}
 Let $A$ be a Hopf algebra. Recall that a (right-right) Yetter-Drinfeld
module over $A$ is a right $A$-comodule and right $A$-module $V$
satisfying the condition, $\forall v \in V$, $\forall a \in A$, 
$$(v \cdot a)_{(0)} \otimes  (v \cdot a)_{(1)} =
v_{(0)} \cdot a_{(2)} \otimes S(a_{(1)}) v_{(1)} a_{(3)}$$
The category of Yetter-Drinfeld modules over $A$ is denoted $\yd_A^A$:
the morphisms are the $A$-linear and $A$-colinear maps. The forgetful functor $\Omega_{A} :\yd_A^{A} \to \mathcal M_{A}$ is  separable if and only if $A$ is cosemisimple and $S^4={\rm id}$, by \cite[Theorem 33]{bi22}, while  $\Omega_{A} :\yd_A^{A} \to \mathcal M_{A}$ is twisted separable if $A$ is cosemisimple by \cite[Proposition 30]{bi22}.
\end{example}

\subsection{Projective dimensions}
The notion of twisted separable functor is motivated by the following result, which proved useful to compare various projective dimensions in \cite{bi22}. The first statement is \cite[Proposition 14]{bi22}, and the second one is the well-known preservation of projective dimension by separable functors, for which we do not have a reference, but whose proof can be deduced from that of  \cite[Proposition 14]{bi22}, using \cite[Lemma 16]{bi22}.

\begin{proposition}\label{prop:pdtsf}
Let $\mathcal{C}$ and $\mathcal{D}$ be abelian categories having enough projective objects, and let $F : \mathcal{C} \to \mathcal{D}$ be a functor. Assume that $F$ is exact and preserves projective objects.
\begin{enumerate}
	\item If $F$ is twisted separable and $\mathcal F$, the corresponding class of objects of  $\mathcal{C}$, contains a generating subclass $\mathcal F_0$ consisting of projective objects, then,  for any object $V$ of $\mathcal{C}$ such that $\pd_{\mathcal{C}}(V)$ is finite, we have $\pd_{\mathcal{C}}(V) = \pd_{\mathcal{D}}(F(V))$. 
\item If $F$ is separable, then,  for any object $V$ of $\mathcal{C}$, we have $\pd_{\mathcal{C}}(V) = \pd_{\mathcal{D}}(F(V))$. 
\end{enumerate}
\end{proposition}

\subsection{Twisted separability of standard type}
The twisted separability data considered in \cite{bi22} were of a special type. For future use and convenience, we formalize this. Let $\varphi : A \to B$ be an algebra automorphism. Then $\varphi$ induces the restriction functor
\begin{align}\varphi_* : {_B \mathcal M} \to {_A\mathcal M} \label{thetastar}\end{align}
sending the left $B$-module $M$ to the left $A$-module $\varphi_*(M)$ having $M$ as underlying vector space and $A$-module structure given by $a\cdot m\ = \varphi(a).m$, for $m \in M$ and $a\in R$. 

If $\theta : R \to R$ is an algebra automorphism,
and $P$ is a free $R$-module, there is a canonical $R$-module isomorphism 
\begin{align}\theta_P : P \to \theta_{*}(P) \label{thetaP}\end{align}
 defined as follows. Fix an $R$-module isomorphism $f : P \to R\otimes V$ for some vector space $V$, and define $\theta_P = f^{-1} \circ ( \theta \otimes {\rm id}_V) \circ f$. It is immediate to check that $\theta_P$ is an $R$-linear isomorphism and is independent of the choice of $f$. 

Of course similar considerations can be done by replacing left modules by right modules.

\begin{definition}
 Let $F : \mathcal C \to {_R\mathcal M}$ be a functor. We say that $F$ is \textsl{twisted separable of standard type} if $F$ is twisted separable with separability datum  $(\Theta, \mathcal F, \theta_{-})$ satisfying the following condition:  there is a generating subclass $\mathcal F_0$ of $\mathcal F$ such for any object $P$ of $\mathcal F_0$ the object $F(P)$ is a free $R$-module, the free $R$-module $R$ is isomorphic to $F(P)$ for some object $P$ of $\mathcal F_0$, and there exists an algebra automorphism $\theta : R\to R$ such that $\Theta = \theta_*$ and $\theta_P$ is as in (\ref{thetaP}) for any object $P$ of $\mathcal F_0$.
\end{definition}

There is of course a completely similar definition for functors with values into right $R$-modules.

\begin{example}
 Let $A$ be a cosemisimple Hopf algebra. The twisted separable functor $\Omega_{A} :\yd_A^{A} \to \mathcal M_{A}$ in Example \ref{ex:ydsep} is of standard type, with the corresponding algebra automorphism $\theta$ defined by $\theta =\psi^{*2}*{\rm id}_A$, where $\psi$ is a modular functional of $A$. See \cite[Section 6]{bi22} for details.
\end{example}

\section{Twisted separability criterion for adjoint functors}\label{sec:tsfa}

In this section we give a characterization of twisted separability for adjoint functors, generalizing Rafael's theorem \cite[Theorem 1.2]{raf}. Our characterization is once a twisted separability datum as been prescribed, while of course, it would be desirable to have a necessary and sufficient condition without such a constraint. However we do not really see how to reach this, and we think that probably this a price to pay for the flexibility of twisted separability.  

\begin{theorem}\label{thm:raftw}
 Let $F : \mathcal C \to \mathcal D$ and $G : \mathcal D \to \mathcal C$ be some functors. Assume that $F$ is left adjoint to $G$, hence that $G$ is right adjoint to $F$, and denote $\eta : 1_{\mathcal C} \to GF$, $\varepsilon : FG \to 1_{\mathcal D}$ the respective unit and counit of the adjunction. 

\begin{enumerate}
 \item  Assume given an autoequivalence $\Theta : \mathcal D\to \mathcal D$ and a generating subclass $\mathcal F$ of objects of $\mathcal C$ together with for any object $P$ of $\mathcal F$, an isomorphism $\theta_P : F(P)\to \Theta F(P)$.
Then $F$ is twisted separable with twisted separability datum $(\Theta, \mathcal F, \theta_{-})$ if and only if there exists a natural transformation $\nu : G\Theta F \to 1_{\mathcal C}$ such that for any object $P$ in $\mathcal F$, one has $\nu_P \circ G(\theta_P)\circ \eta_P = {\rm id}_P$.

\item Assume given an autoequivalence $\Theta : \mathcal C\to \mathcal C$ and a generating subclass $\mathcal F$ of objects of $\mathcal D$ together with for any object $P$ of $\mathcal F$, an isomorphism $\theta_P : G(P)\to \Theta G(P)$. Fix a quasi-inverse equivalence $\Theta^{-1}$ to $\Theta$ together with natural isomorphisms $u :1_{\mathcal C}\to \Theta\Theta^{-1}$ and $v : \Theta^{-1}\Theta\to 1_{\mathcal C}$ forming an adjoint equivalence. 
Then $G$ is twisted separable with twisted separability datum $(\Theta, \mathcal F, \theta_{-})$ if and only if there exists a natural transformation $\xi : 1_{\mathcal D}  \to F\Theta^{-1} G$ such that for any object $P$ in $\mathcal F$, one has $\varepsilon_P \circ F(v_{G(P)} \circ \Theta^{-1}(\theta_P))\circ \xi_P = {\rm id}_P$.
\end{enumerate}
\end{theorem}

\begin{proof}
 (1) Assume first that $F : \mathcal C\to \mathcal D$ is twisted separable with twisted separability datum $(\Theta, \mathcal F, \theta_{-})$, so that we are given   a natural morphism $$\textbf{M}_{-,-} : \Hom_{\mathcal{D}}(F(-),\Theta F(-)) \to \Hom_{\mathcal{C}}(-,-)$$ such that for any object $P$ of $\mathcal{F}$, we have $\textbf{M}_{P,P}(\theta_P)= {\rm id}_P$. Define, for an object $X$ of $\mathcal C$, $\nu_X : G\Theta F(X)\to X$ by
$$\nu_X = \textbf{M}_{G\Theta F(X),X}(\varepsilon_{\Theta F(X)})$$
For a morphism $\beta : X\to Y$, we have, using the naturality of $\textbf{M}_{-,-}$ and of $\varepsilon$,
\begin{align*}
 \beta \circ \nu_X & = \beta \circ \textbf{M}_{G\Theta F(X),X}(\varepsilon_{\Theta F(X)}) = 
 \textbf{M}_{G\Theta F(X),Y}(\Theta F(\beta)\circ \varepsilon_{\Theta F(X)})
 \\
& = \textbf{M}_{G\Theta F(X),Y}(\varepsilon_{\Theta F(Y)} \circ FG\Theta F(\beta)) = \textbf{M}_{G\Theta F(Y),Y}(\varepsilon_{\Theta F(Y)}) \circ G\Theta F(\beta) \\
& = \nu_Y \circ G \Theta F(\beta)
\end{align*}
This shows that $\nu$ above indeed defines a natural transformation $G\Theta F \to 1_{\mathcal C}$. For an object $P$ of $\mathcal F$, we have, using the naturality of $\textbf{M}_{-,-}$, of $\varepsilon$, the adjunction property and our assumption,
\begin{align*}
 \nu_P\circ G(\theta_P)\circ \eta_P & = \textbf{M}_{G\Theta F(P),P}(\varepsilon_{\Theta F(P)}) \circ G(\theta_P)\circ  \eta_P =  \textbf{M}_{P,P}(\varepsilon_{\Theta F(P)} \circ FG(\theta_P)\circ  F(\eta_P))\\ 
& = \textbf{M}_{P,P}(\theta_P \circ \varepsilon_{F(P)} \circ  F(\eta_P)) = M(\theta_P)= {\rm id}_P
\end{align*}
This shows that $\nu$ satisfies the announced condition.

Conversely, start with natural transformation $\nu : G\Theta F \to 1_{\mathcal C}$ such that for any object $P$ in $\mathcal F$, one has $\nu_P \circ G(\theta_P)\circ \eta_P = {\rm id}_P$, and define for objects $V$, $W$ of $\mathcal C$,
\begin{align*}  
\textbf{M}_{V,W} :  \Hom_{\mathcal{D}}(F(V),\Theta F(W)) &\longrightarrow \Hom_{\mathcal{C}}(V,W) \\
f &\longmapsto \nu_W \circ G(f) \circ \eta_V
\end{align*}
Consider morphisms $\alpha : V'\to V$, $\beta : W \to W'$ in $\mathcal{C}$ and a morphism $f : F(V) \to \Theta F(W)$ in $\mathcal{D}$. We have, using the naturality of $\eta$ and $\nu$,
\begin{align*}
\beta \circ {\bf M}_{V,W}(f) \circ \alpha & = \beta \circ \nu_W \circ G(f) \circ \eta_V \circ \alpha
= \nu_{W'} \circ G\Theta F(\beta) \circ G(f) \circ  GF(\alpha)\circ \eta_{V'}\\
& = \nu_{W'} \circ G(\Theta F(\beta) \circ f \circ  F(\alpha))\circ \eta_{V'} ={\bf M}_{V',W'}(\Theta F(\beta) \circ f \circ F(\alpha)) 
\end{align*}
and this shows that $\textbf{M}_{-,-}$ is natural. We also have, by our assumption, $\textbf{M}_{P,P}(\theta_P) = \nu_P \circ G(\theta_p) \circ \eta_P={\rm id}_P$, and we conclude that $F$ is twisted separable with  twisted separability datum $(\Theta, \mathcal F, \theta_{-})$.

(2) Assume that $G : \mathcal D\to \mathcal C$ is twisted separable with twisted separability datum $(\Theta, \mathcal F, \theta_{-})$, so that we are given   a natural morphism $$\textbf{M}_{-,-} : \Hom_{\mathcal{C}}(G(-),\Theta G(-)) \to \Hom_{\mathcal{D}}(-,-)$$ such that for any object $P$ of $\mathcal{F}$, we have $\textbf{M}_{P,P}(\theta_P)= {\rm id}_P$. Define, for an object $X$ of $\mathcal C$, $\xi_X : X\to F\Theta^{-1} G(X)$ by
$$\xi_X = \textbf{M}_{X,F\Theta^{-1} G(X)}\left(\Theta(\eta_{\Theta^{-1} G(X)})\circ u_{G(X)}\right)$$
For a morphism $\beta : X\to Y$, we have, using the naturality of $\textbf{M}_{-,-}$, of $\eta$ and $u$,
\begin{align*}
 F\Theta^{-1}G(\beta) \circ \xi_X & = F\Theta^{-1}G(\beta) \circ  \textbf{M}_{X,F\Theta^{-1} G(X)}\left(\Theta(\eta_{\Theta^{-1} G(X)})\circ u_{G(X)}\right) \\
&= \textbf{M}_{X,F\Theta^{-1} G(Y)}\left(\Theta GF\Theta^{-1}G(\beta) \circ \Theta(\eta_{\Theta^{-1} G(X)})\circ u_{G(X)}\right)
 \\ 
& = \textbf{M}_{X,F\Theta^{-1} G(Y)}\left(\Theta(\eta_{\Theta^{-1}G(Y)}) \circ \Theta\Theta^{-1}G(\beta)) \circ u_{G(X)}\right)
 \\  &=\textbf{M}_{X,F\Theta^{-1} G(Y)}\left(\Theta(\eta_{\Theta^{-1}G(Y)}) \circ u_{G(Y)} \circ G(\beta) \right) \\
& = \textbf{M}_{Y,F\Theta^{-1} G(Y)}\left(\Theta(\eta_{\Theta^{-1}G(Y)}) \circ u_{G(Y)} \right) \circ \beta \\
& = \xi_Y \circ \beta
\end{align*}
and this shows that $\xi$ is a natural transformation. We have, for $P \in \mathcal C$, 
\begin{align*}
\varepsilon_P &\circ F(v_{G(P)} \circ \Theta^{-1}(\theta_P))\circ \xi_P \\ 
&=  \varepsilon_P \circ F(v_{G(P)} \circ \Theta^{-1}(\theta_P))\circ  \textbf{M}_{P,F\Theta^{-1} G(P)}\left(\Theta(\eta_{\Theta^{-1} G(P)})\circ u_{G(P)}\right)\\
& = \textbf{M}_{P,P}\left( \Theta G(\varepsilon_P) \circ \Theta GF(v_{G(P)}) \circ \Theta G F\Theta^{-1}(\theta_P)\circ  \Theta(\eta_{\Theta^{-1} G(P)})\circ u_{G(P)}\right) \\
& =\textbf{M}_{P,P}\left( \Theta G(\varepsilon_P) \circ \Theta GF(v_{G(P)}) \circ \Theta(\eta_{\Theta^{-1} \Theta G(P)}) \circ \Theta \Theta^{-1}(\theta_P)  \circ u_{G(P)}\right) \\
&  =\textbf{M}_{P,P}\left( \Theta G(\varepsilon_P)   \circ \Theta(\eta_{G(P)}) \circ  \Theta (v_{G(P)}) \circ u_{\Theta G(P)} \circ \theta_P\right) \\
& =\textbf{M}_{P,P}(\theta_P) ={\rm id}_P
\end{align*}
where we have used naturality of $\textbf{M}_{-,-}$ and the adjoint properties.

Conversely, assume we are given a natural transformation $\xi : 1_{\mathcal D}  \to F\Theta^{-1} G$ such that for any object $P$ in $\mathcal F$, one has $\varepsilon_P \circ F(v_{G(P)} \circ \Theta^{-1}(\theta_P))\circ \xi_P = {\rm id}_P$.
Define for objects $X$, $Y$ of $\mathcal D$,
\begin{align*}  
\textbf{M}_{X,Y} :  \Hom_{\mathcal{C}}(G(X),\Theta G(Y)) &\longrightarrow \Hom_{\mathcal{D}}(X,Y) \\
f &\longmapsto \varepsilon_Y \circ F\left(v_{G(Y)} \circ\Theta^{-1}(f)\right) \circ \xi_X
\end{align*}
We have  ${\bf M}_{P,P}(\theta_P)={\rm id}_P$ for any object $P$ in $\mathcal F$, by assumption. 
Consider morphisms $\alpha : X'\to X$, $\beta : Y \to Y'$ in $\mathcal{D}$ and a morphism $f : G(X) \to \Theta G(Y)$ in $\mathcal{C}$. We have, using the naturality of $\varepsilon$, $\xi$ and $v$,
\begin{align*}
\beta \circ {\bf M}_{X,Y}(f) \circ \alpha & = \beta \circ \varepsilon_Y \circ F\left(v_{G(Y)}\circ\Theta^{-1}(f)\right) \circ \xi_X \circ \alpha \\
&= \varepsilon_{Y'} \circ F\left(G(\beta) \circ v_{G(Y)}\circ \Theta^{-1}(f) \circ  \Theta^{-1}G(\alpha)\right) \circ \xi_{X'}\\
&= \varepsilon_{Y'} \circ F\left(v_{G(Y')} \circ \Theta^{-1}\Theta G(\beta) \circ \Theta^{-1}(f) \circ  \Theta^{-1}G(\alpha)\right) \circ \xi_{X'}\\
& = \varepsilon_{Y'} \circ F\left(v_{G(Y')} \circ \Theta^{-1}\left(\Theta G(\beta) \circ (f) \circ  G(\alpha)\right)\right) \circ \xi_{X'}\\
&=  {\bf M}_{X',Y'}\left(G\Theta(\beta)\circ f \circ G(\alpha)\right) 
\end{align*}
This shows that $\textbf{M}_{-,-}$ is natural, and we conclude that $G$ is twisted separable with  twisted separability datum $(\Theta, \mathcal F, \theta_{-})$.
\end{proof}

\begin{remark}
 The above characterization of twisted separability is more technical for right adjoint functors than for left adjoint functors, but this is not surprising because of the asymmetry in the definition of twisted separability.
\end{remark}

\section{Application to Hopf-Galois objects}\label{sec:apphg}

In this section we apply the twisted separability criterion of Section \ref{sec:tsfa} to Hopf-Galois objects, which will lead to the proof of Theorem \ref{thm:cdhgtt} in the introduction.

\subsection{The canonical functor and its separability} We begin by associating a functor to any left comodule algebra, as follows.

\begin{definition}
 Let $A$ be a Hopf algebra and let $R$ be a left $A$-comodule algebra. The \textsl{canonical functor associated to $R$} is the functor
\begin{align*}
 F_R : {_A \mathcal M} &\longrightarrow {_R\mathcal M}_R \\
V &\longmapsto V \odot R
\end{align*}
where $V\odot R$ is $V\otimes R$ as a vector space, and has $R$-bimodule structure given by
$$x\cdot (v\otimes y)\cdot z = x_{(-1)}\cdot v \otimes x_{(0)}yz, \quad x,y,z \in R, \ v\in V$$ 
\end{definition}

We are interested in the separability or twisted separability of the above canonical functor. In general it is unclear to us what can be said about this question, and we will specialize to the case of Hopf-Galois objects. Recall that if $A$ is a Hopf algebra, a left $A$-Galois object is a (non-zero) left $A$-comodule algebra such that the canonical map
\begin{align*}
 {\rm can} : R\otimes R & \longrightarrow A \otimes R \\
x\otimes y &\longmapsto x_{(-1)} \otimes x_{(0)}y
\end{align*}
 is bijective. See \cite{scsurv} for a general overview of the theory of Hopf-Galois objects, which in particular serve to construct monoidal functors between categories of comodules.  

The map $ \kappa : A \to R\otimes R$ defined by $\kappa (a) = {\rm can}^{-1}(a\otimes 1)$ defines then an algebra map $ A \to R\otimes R^{\rm op}$, for which we use a Sweedler type notation $\kappa(a)=a^{[1]} \otimes a^{[2]}$. By construction of $\kappa$, the following identities hold, for $x\in R$ and $a\in A$:
\begin{align}
 \label{eq:kappa} x_{(-1)}^{\quad \ [1]} \otimes x_{(-1)}^{\quad \ [2]}x_{(0)} = x\otimes 1, \quad 
a^{[1]}_{\ (-1)} \otimes a^{[1]}_{\ (0)}a^{[2]} = a\otimes 1
\end{align}
The above algebra map defines in the usual way the restriction functor
$$\kappa_* : {_A\mathcal M} \to {_R\mathcal M}_R$$  
that associates to an $R$-bimodule $M$ the left $A$-module $\kappa_*(M)$ having $M$ as underlying vector space and left $A$-action, called the Miyashita-Ulbrich action, defined by $a\cdot x = a^{[1]}.x.a^{[2]}$ 

The following result is well-known.

\begin{proposition}\label{prop:adjcan}
 Let $A$ be a Hopf algebra and let $R$ be a left $A$-Galois object. Then $(F_R, \kappa_*)$ form a pair of adjoint functors, the unit and counit of the adjonction being given by the following morphisms, for $V$ a left $A$-module and $M$ an $R$-bimodule
 \begin{align*}
  \eta_V : V &\longmapsto \kappa_*(V\odot R) \quad \quad \varepsilon_M : \kappa_*(M)\odot R \to M \\
v &\longmapsto v\otimes 1 \quad \quad \quad \quad  \quad \quad\quad \quad m\otimes x \longmapsto m\cdot x 
 \end{align*}
\end{proposition}

\begin{proof}
 That $\eta_V$ and $\varepsilon_M$ are morphisms respectively in ${_A \mathcal M}$ and ${_R\mathcal M}_R$ follow directly from the identities (\ref{eq:kappa}), and the adjunction identities are immediate to check.
\end{proof}

Using Rafael's theorem, We thus can characterize when the canonical functor of a left Galois object is separable, as follows

\begin{theorem}\label{thm:sephg}
 Let $A$ be a Hopf algebra and let $R$ be a left $A$-Galois object. Then the canonical functor $F_R : {_A \mathcal M} \rightarrow {_R\mathcal M}_R$ is separable if and only if there exists a linear map $\psi : R\to k$ such that $\psi(1)=1$ and $\psi$ is $A$-linear with respect to the Miyashita-Ulbrich action on $R$, i.e. $\psi ( a^{[1]}.x.a^{[2]})= \varepsilon(a)\psi(x)$, for any $x\in R$ and $a\in A$.
\end{theorem}

\begin{proof}
 By Rafael's theorem and Proposition \ref{prop:adjcan}, the separability of $F_R$ is equivalent to the existence of a natural transformation
$$\nu : \kappa_*(\bullet \odot R) \to 1_{{_A\mathcal M}}$$
such that $\nu_V\circ\eta_V= {\rm id}_V$ for any left $A$-module $V$. If such a $\nu$ exists, put $\psi =\nu_k$ for the trivial object $k$. It is immediate that $\psi$ satisfies the above conditions. Starting with an $A$-linaear map $\psi$, define $\nu_V = {\rm id}_V\otimes \psi$. We have, for $a\in A$, $v\in V$ and $x\in R$,
\begin{align*}
 \nu_V(a\cdot(v\otimes x)) & = \nu_V(a^{[1]}. (v\otimes x). a^{[2]}) 
 =  \nu_V\left(a^{[1]}_{\ (-1)}.v \otimes a^{[1]}_{\ (0)} x a^{[2]}\right) \\
& = \nu_V\left( a_{(1)}.v \otimes a_{(2)}^{\ \ [1]} x a_{(2)}^{\ \ [2]} \right) =  \psi( a_{(2)}^{\ \ [1]} x a_{(2)}^{\ \ [2]}) a_{(1)}.v \\
& = \psi(x)a.v= a.\nu_V(v\otimes x)
\end{align*}
where we have used the identity \cite[Remark 3.4]{schneider}
\begin{align}\label{eq:idkappa3}
 a^{[1]}_{\ (-1)} \otimes a^{[1]}_{\ (0)} \otimes a^{[2]} = a_{(1)} \otimes a_{(2)}^{\ \ [1]} \otimes a_{(2)}^{\ \ [2]}
\end{align}
Hence $\nu_V$ is $A$-linear, and since $\nu_V\circ \eta_V = {\rm id}_V$ and $\nu$ defines a natural transformation, we conclude that $F_R$ is separable.
\end{proof}

\subsection{Cogroupoids}
We now want to generalize Theorem \ref{thm:sephg} to  the twisted separable context. While its proof essentially only needed   the identities (\ref{eq:kappa}) and (\ref{eq:idkappa3}), the proof of our generalization will need some more involved ones, and it will be convenient to move to the context of cogroupoids \cite{bic14}, that we recall now.  

First recall that a cocategory (of algebras over $k$) $\coc$ consists of:

\noindent
$\bullet$ a set of objects ${\rm ob}(\coc)$;

\noindent
$\bullet$ For any $X,Y \in {\rm ob}(\coc)$, an algebra 
$\coc(X,Y)$; 

\noindent
$\bullet$ For any $X,Y,Z \in {\rm ob}(\coc)$, algebra morphisms
$$\Delta_{X,Y}^Z : \coc(X,Y) \longrightarrow \coc(X,Z) \otimes \coc(Z,Y)
\quad {\rm and} \quad \varepsilon_X : \coc(X,X) \longrightarrow \C$$
such that some natural coassociativity and counit diagrams, dual to the usual associativity and unit diagrams in a category, commute. In particular $\coc(X,X)$ is a bialgebra for any object $X$.

A cogroupoid  $\coc$ consists of a cocategory $\coc$ together
with, for any $X,Y \in {\rm ob}(\coc)$, linear maps
$$S_{X,Y} : \coc(X,Y) \longrightarrow \coc(Y,X)$$
such that natural diagrams (dual to the invertibility diagrams in a groupoid) commute.
In particular $\coc(X,X)$ is a Hopf algebra for any object $X$.
A cogroupoid is said to be connected if for any $X,Y \in {\rm ob}(\coc)$, the algebra
$\mathcal C(X,Y)$ is non-zero.

Let us now write down the axioms using Sweedler's notation for cocategories and cogroupoids.
Let $\coc$ be a cocategory. For $a^{XY} \in \coc(X,Y)$, we write 
$$\Delta_{X,Y}^Z(a^{XY})= a_{(1)}^{XZ} \otimes a_{(2)}^{ZY}$$
The cocategory axioms are
$$(\Delta_{X,Z}^T \otimes 1)\circ \Delta_{X,Y}^Z(a^{XY}) = 
a_{(1)}^{XT} \otimes a_{(2)}^{TZ} \otimes a_{(3)}^{ZY}= 
(1 \otimes \Delta_{T,Y}^Z)\circ \Delta_{X,Y}^T(a^{XY})
$$
$$\varepsilon_X(a_{(1)}^{XX}) a_{(2)}^{XY} = a^{XY} =
\varepsilon_Y(a_{(2)}^{YY})  a_{(1)}^{XY}$$
and the additional cogroupoid axioms are
$$S_{X,Y}(a_{(1)}^{XY}) a_{(2)}^{YX} = \varepsilon_X(a^{XX})1 =
a_{(1)}^{XY} S_{Y,X}(a_{(2)}^{YX})
$$

It is known \cite[Theorem 2.11]{bic14} that if $A$ is a Hopf algebra and $R$ is a left $A$-Galois object, there exists a connected cogroupoid $\coc$ with two objects $X$, $Y$ such that $A =\coc(X,X)$ and $R=\coc(X,Y)$. In this setting, we have 
for $a=a^{XY}\in \coc(X,Y)=R$, $$\kappa(a) = a^{[1]} \otimes a^{[2]} =   a_{(1)}^{X,Y} \otimes S_{Y,X}(a_{(2)}^{YX}),$$ and 
Identities (\ref{eq:kappa}) are, for $a^{XY} \in \coc(X,Y)=R$, $a^{XX}\in \coc(X,X)=A$,  
$$a_{(1)}^{XY}\otimes S_{Y,X}(a_{(2)}^{YX})a^{XY}_{(3)} = a^{XY}\otimes 1, \quad a_{(1)}^{XX}\otimes a_{(2)}^{XY}S_{Y,X}(a^{YX}_{(3)}) = a^{XX}\otimes 1$$
while  (\ref{eq:idkappa3}) is, for $a= a^{XY} \in R$, 
$$
a^{[1]}_{\ (-1)} \otimes a^{[1]}_{\ (0)} \otimes a^{[2]} = a_{(1)}^{XX} \otimes a_{(2)}^{XY} \otimes S_{Y,X}(a_{(3)}^{YX}) = 
a_{(1)} \otimes a_{(2)}^{\ \ [1]} \otimes a_{(2)}^{\ \ [2]}$$
We also record the following useful facts \cite[Proposition 2.13]{bic14}: for $X,Y,Z\in {\rm ob}(\coc)$, the map $S_{X,Y} : \coc(X,Y)\to \coc(Y,X)$ is an anti-algebra morphism, and for $a^{YX} \in \coc(X,Y)$, we have 
\begin{align}\Delta_{XY}^Z(S_{Y,X}(a^{YX}))= S_{Z,X}(a_{(2)}^{ZX})\otimes S_{Y,Z}(a_{(1)}^{YZ})\label{eq:anticoSYX}
 \end{align}

\subsection{Twisted traces and twisted separability for the canonical functor} We begin this subsection with the following classical definition, already evoked in the introduction.

\begin{definition}
Let $R$ be an algebra. A linear map $\psi : R \to k$ is said to be a \textsl{twisted trace} if there exists an algebra automorphism $\sigma : R \to R$ such that $\psi(xy)=\psi(y\sigma(x))$ for any $x,y\in R$. In that case $\psi$ is said to be a twisted trace with respect to $\sigma$.  A trace is a twisted trace with respect to the identity. A twisted trace $\psi$ is said to be \textsl{unital} if $\psi(1)=1$.
\end{definition}

We also will need the following definition.

\begin{definition}
 Let $A$ be a Hopf algebra and let $V,W$ be left $A$-comodules. A linear map $f: V\to W$ is said to be \textsl{semi-$A$-colinear} if there exists a Hopf algebra automorphism $\rho : A\to A$ such that
$\rho(v_{(-1)}) \otimes f(v_{(0)}) = f(v)_{(-1)} \otimes f(v)_{(0)}$ for any $v  \in V$. In that case we say that $f$ is left semi-$A$-colinear with respect to $\rho$.
\end{definition}

\begin{example}\label{ex:semico}
 Let $\coc$ be a cogroupoid and let $X,Y\in {\rm ob}(\coc)$. It follows from (\ref{eq:anticoSYX}) that  $S_{Y,X}S_{X,Y} : \coc(X,Y) \to \coc(X,Y)$ is left semi-colinear with respect to $S_{XX}^2$.
\end{example}

Twisted traces are related with a twisted version of Miyashita-Ulbrich invariance in Theorem \ref{thm:sephg}, as follows.

\begin{lemma}\label{lem:tt-my}
 Let $A$ be a Hopf algebra with bijective antipode, let $R$ be a left Hopf-Galois object over $A$ and let $\psi : R\to k$ be a linear map. The following assertions are equivalent:
\begin{enumerate}
 \item $\psi$ is a twisted trace, with respect to an automorphism $\sigma : R\to R$;
\item There exists and algebra automorphism $\theta : R \to R$ such that for any $a\in A$ and $x\in R$, we have $\psi(\theta(a^{[1]})xa^{[2]})= \varepsilon(a)\psi(x)$. 
\end{enumerate}
In this case, if  $\mathcal C$ be a connected groupoid with objects $X,Y$ such that $A=\mathcal C(X,X)$ and $R = \mathcal C(X,Y)$, the automorphisms $\sigma$ and $\theta$ are related by $\sigma \theta =  S_{Y,X}S_{X,Y}$, and $\sigma$ is semi-$A$-colinear if and only $\theta$ is.
\end{lemma}

\begin{proof}
Let $\mathcal C$ be a connected groupoid with objects $X,Y$ such that $A=\mathcal C(X,X)$ and $R = \mathcal C(X,Y)$, so that for $a=a^{XX} \in A$, we have $a^{[1]} \otimes a^{[2]} = a_{(1)}^{XY} \otimes S_{Y,X}(a^{YX}_{(2)})$.   

Assume first that $\psi$ is a twisted trace with respect to an automorphism $\sigma : R\to R$, and let $\theta = \sigma^{-1} S_{Y,X} S_{X,Y} : R \to R$. We have, for $x\in R$ and $a= a^{XX}\in A$
\begin{align*}
 \psi\left (\theta(a^{[1]}) x a^{[2]}\right)& = \varphi\left (\theta(a_{(1)}^{XY}) x S_{Y,X}(a^{YX}_{(2)})\right) 
= \psi\left (\sigma^{-1}(S_{Y,X}S_{X,Y}(a_{(1)}^{XY})) x S_{Y,X}(a^{YX}_{(2)})\right) \\
& = \psi\left ( x S_{Y,X}(a^{YX}_{(2)})S_{Y,X}S_{X,Y}(a_{(1)}^{XY})\right) 
=  \psi\left ( x S_{Y,X}(S_{X,Y}(a_{(1)}^{XY})a^{YX}_{(2)}))\right) \\
&= \varepsilon_{X,X}(a^{XX}) \psi(x) =\varepsilon(a)\psi(x)
\end{align*}
 Conversely, start $\theta$ satisfying the above condition in (2) and let $\sigma = S_{Y,X}S_{X,Y}\theta^{-1}$. For $x= a^{XY}\in R$ and $y\in R$, we have 
\begin{align*}
 \psi(\theta(x)y) & = \psi(\theta(a^{XY})y)=  \psi\left(\theta(a^{XY}_{(1)})yS_{Y,X}S_{X,Y}(a_{(3)}^{XY})S_{Y,X}(a_{(2)}^{YX})\right) \\
& = \varepsilon_{X,X}(a_{(1)}^{XX})\psi\left(yS_{Y,X}S_{X,Y}(a_{(2)}^{XY})\right) = 
\psi\left(yS_{Y,X}S_{X,Y}(x)\right)
\end{align*}
and hence 
$$\psi(xy) = \psi(\theta(\theta^{-1}(x)y) = \psi(y S_{Y,X}S_{X,Y}\theta^{-1}(x)) = \psi(y\sigma(x))$$
which indeed shows that $\psi$ is a twisted trace with respect to $\sigma$. Clearly our construction relates $\sigma$ and $\theta$ as in the statement, and the last assertion follows from Example \ref{ex:semico}.
\end{proof}

\begin{theorem}\label{thm:tshg}
 Let $A$ be a Hopf algebra with bijective antipode and let $R$ be a left $A$-Galois object. Assume that $R$ admits a unital  twisted trace with respect to a semi-A-colinear automorphism of $R$. Then the canonical functor ${_A\mathcal M} \to {_R\mathcal M}_R$ is twisted separable.
\end{theorem}

\begin{proof}
 Let $\mathcal C$ be a connected groupoid with objects $X,Y$ such that $A=\mathcal C(X,X)$ and $R = \mathcal C(X,Y)$, so that for $a=a^{XX} \in A$, we have $a^{[1]} \otimes a^{[2]} = a_{(1)}^{XY} \otimes S_{Y,X}(a^{YX}_{(2)})$.  
Let $\psi : R\to k$ be a unital twisted trace with respect to a left $A$-semi-colinear automorphism  $\sigma : R\to R$, and let $\gamma : A\to A$ be the corresponding Hopf algebra automorphism : $\Delta_{X,Y}^X \circ \sigma = (\gamma \otimes \sigma) \circ \Delta_{X,Y}^X$. We put, as in the proof of Lemma \ref{lem:tt-my}, $\theta = \sigma^{-1}S_{Y,X}S_{X,Y}$, so that  $\psi\left (\theta(a^{[1]}) x a^{[2]}\right)=\varepsilon(a)\psi(x)$ for any $a\in A$ and $x\in R$. Define also the Hopf algebra automorphism $\rho : A\to A$ by $\rho = \gamma^{-1} S^2$ (with $S^2=S^2_{X,X}$), so that $\Delta_{X,Y}^X \circ \theta = (\rho \otimes \theta)\circ \Delta_{X,Y}^X$. 

We have to construct a twisted separability datum $(\mathcal F, \Theta, \delta_{-})$ for $F_R$. We take $\mathcal F$ the class of free $A$-modules and $\Theta =(\theta\otimes {\rm id}_R)_*$, where we identify ${_R\mathcal M}_R$ and ${_{R\otimes R^{\rm op}}\mathcal M}$ in the usual way. For a free $A$-module $P = A\otimes V$, we define 
\begin{align*}
 \delta_P : (A\otimes V) \odot R &\longrightarrow  (\theta\otimes {\rm id})_*(P\odot R) \\
a^{XX} \otimes v \otimes x & \longmapsto \rho(a^{XX}_{(1)}) \otimes v \otimes \theta(a_{(2)}^{XY})S_{Y,X}(a_{(3)}^{YX})x 
\end{align*}
The map $\delta_P$ is the unique one making the following diagram commutative:
$$\xymatrixcolsep{5pc}
\xymatrix{ 	
	(A\otimes V) \odot R \ar[r]^-{\delta_P} & (\theta\otimes {\rm id})_*((A\otimes V) \odot R)  \\
	R\otimes V \otimes R  \ar[u]^-{{\rm can}} \ar[r]^-{\theta \otimes {\rm id}_V \otimes {\rm id}_R} & (\theta\otimes {\rm id})_*(R\otimes V \otimes R) \ar[u]^-{{\rm can}}
}$$ 
where ${\rm can}$ is the obvious adaptation the canonical map, inserting $V$ in the middle.
Hence, since all the involved map are $R$-bimodule isomorphisms, so is $\delta_P$.

Define for any left $A$-module $V$,
\begin{align*}
 \psi_V : V\otimes R & \longrightarrow V \\
v \otimes x & \longmapsto \psi(x_{(0)}) S^{-1}(x_{(-1)})\cdot v
\end{align*}
We wish to show that $\psi_V : \kappa_*((\theta\otimes {\rm id}_R)_*(V \odot R)) \to V$ is left $A$-linear. Let $a=a^{XX} \in A$, $v\in V$ and $x=b^{XY}\in R$. We have
\begin{align*}
 \psi_V(a^{XX}&\cdot (v\otimes x))  = \psi_V\left( \theta(a_{(1)}^{XY})_{(-1)}\cdot v \otimes \theta(a_{(1)}^{XY})_{(0)} x S_{Y,X}(a_{(2)}^{YX})\right) \\
& = \psi_V\left( \rho(a_{(1)}^{XX})\cdot v \otimes \theta(a_{(2)}^{XY}) x S_{Y,X}(a_{(3)}^{YX})\right) \\
&= \psi\left(\theta(a_{(3)}^{XY}) x_{(0)} S_{YX}(a_{(4)}^{YX})\right) S^{-1}\left(\rho(a_{(2)}^{XX})x_{(-1)}S_{XX}(a_{(5)}^{XX})\right) \cdot \rho(a^{XX}_{(1)})\cdot v \\
& = \psi\left(\theta(a_{(3)}^{XY}) x_{(0)} S_{YX}(a_{(4)}^{YX})\right) 
a_{(5)}^{XX} S^{-1}(x_{(-1)})
\rho (S_{XX}^{-1}(a_{(2)}^{XX}))  \rho(a^{XX}_{(1)})\cdot v \\
& =  \psi\left(\theta(a_{(1)}^{XY}) x_{(0)} S_{YX}(a_{(2)}^{YX})\right) 
a_{(3)}^{XX} S^{-1}(x_{(-1)})\cdot v\\
& = \varepsilon(a_{(1)}^{XX})\psi(x_{(0)}) 
a_{(2)}^{XX} S^{-1}(x_{(-1)})\cdot v = a^{XX}\cdot \psi_V(v \otimes x)
\end{align*}
Hence $\psi_V$ is $A$-linear, and clearly this defines a natural transformation. For a free $A$-module $P=A\otimes V$ and $a=a^{XX}\in A$ and $v \in V$, we have
\begin{align*}
 \psi_P \circ \delta_P\circ \eta_P(a\otimes v) & = \psi_P\left( \rho(a^{XX}_{(1)}) \otimes v \otimes \theta(a_{(2)}^{XY})S_{YX}(a_{(3)}^{YX})x)\right) \\
& = \psi\left( \theta(a_{(3)}^{XY})S_{YX}(a_{(4)}^{YX})x)\right)S_{XX}^{-1}\left( \rho(a_{(2)}^{XX}) S_{XX}(a_{(5)}^{XX})\right)\rho(a^{XX}_{(1)})\otimes v \\
&= a^{XX}\otimes v
\end{align*}
Hence we have 
$\psi_P \circ \delta_P\circ \eta_P ={\rm id}_P$, and the properties of the natural transformation $\psi$
allow us to apply Theorem \ref{thm:tshg} to the left adjoint functor $F_R$, to conclude that it is indeed twisted separable.
\end{proof}

\begin{corollary}\label{cor:trace}
  Let $A$ be a Hopf algebra with bijective antipode and let $R$ be a left $A$-Galois object. If $R$ admits a unital trace, then the canonical functor ${_A\mathcal M} \to {_R\mathcal M}_R$ is twisted separable.
\end{corollary}

\begin{proof}
A trace on $R$ is a twisted trace with respect to the identity of $R$, which is semi-colinear with respect to the identity of $A$, hence the result follows from Theorem \ref{thm:tshg}. 
\end{proof}

\begin{example}
 The Weyl algebra $$A_1(k)=k\langle x,y \ | \ xy-yx=1\rangle$$ is a Hopf-Galois object over $k[x,y]$, and the canonical functor ${_{k[x,y]}\mathcal M} \to {_{A_1(k)}\mathcal M}_{A_1(k)}$ is not separable . However, if $k$ has characteristic zero,  it is twisted separable 
\end{example}

\begin{proof}
 Let $\mathcal C$ be the cogroupoid having $\{0,1 \}$ as objects, with for $i,j \in\{0,1\}$ 
$$\coc(i,j) =k\langle x_{ij},y_{ij} \ | \ x_{ij}y_{ij}-y_{ij}x_{ij}=i-j\rangle$$
and structural maps defined by
$$\Delta_{ij}^k(x_{ij}) = 1\otimes x_{kj} + x_{ik}\otimes 1, \ \Delta_{ij}^k(y_{ij}) = 1\otimes y_{kj} + y_{ik}\otimes 1$$
$$\varepsilon_{ii}(x_{ii}) =0=\varepsilon(y_{ii}), \quad S_{ij}(x_{ij}) = -x_{ji}, \quad S_{ij}(y_{ij}) = -y_{ji}$$
We have $k[x,y]= \coc(1,1)$ and $A_{1}(k) = \coc(1,0)$. Since $S_{ij}S_{ji}={\rm id}$, we see from Theorem \ref{thm:sephg} and Lemma  \ref{lem:tt-my}
that the separability of the canonical functor would imply the existence of a unital trace on $A_{1}(k)$, which is impossible since $xy-yx=1$.

However, by \cite[Corollary 2.4]{ekrs}, there exists a unital twisted trace on $A_1(k)=\coc(1,0)$, with respect to the automorphism defined by $\alpha_t(x)=t^{-1}x$ and $\alpha_t(y)=ty$ for $t\in k^*$, $t\not=1$. It is easily seen that $\alpha_t$ is semi-$\coc(1,1)$-colinear with respect to the automorphism of $\coc(1,1)$ defined by the same formula as $\alpha_t$, and hence the twisted separability of the canonical functor follows from Theorem \ref{thm:tshg}.
\end{proof}

\subsection{Application to cohomological dimension} We now prove Theorem \ref{thm:cdhgtt}.  Let $A$ be a Hopf algebra with bijective antipode and let $R$ be a left $A$-Galois object that admits a twisted trace relative to an algebra automorphism of $R$ which is left $A$-semi-colinear. The canonical functor $F_R : {_A\mathcal M} \to {_R\mathcal M}_R$ is twisted separable by Theorem \ref{thm:tshg}, with the corresponding class $\mathcal F$ consisting of free $A$-modules, and it is obvious that $F_R$ is exact, and preserve projective objects since it is a left adjoint functor. Assuming that $\cd(A)$ is finite, we therefore apply Proposition \ref{prop:pdtsf}
to obtain
\[
\cd(A) = \pd_A(k) = \pd_{_R\mathcal M_R}(F_R(k))=\pd_{_R\mathcal M_R}(R) = \cd(R)
\]
which finishes the proof of Theorem \ref{thm:cdhgtt}.

\begin{remark}
Let $A$ be a finite-dimensional Hopf algebra and let $R$ be left $A$-Galois object. If the base field $k$ has characteristic zero, then $R$ always has a  unital trace since it is a finite-dimensional algebra, so it follows from Corollary \ref{cor:trace} that
 the canonical functor ${_A\mathcal M} \to {_R\mathcal M}_R$ is twisted separable. However, if $A$ is not semisimple, then  $\cd(A)=\infty$ (this is because Hopf algebras are self-injective algebras),  and if $R$ is semisimple, we have $\cd(R)=0 < \cd(A)$. Such a situation  holds, for example,  for Taft algebras, which have matrix algebras as Galois objects, see \cite{mas}.
Thus the assumption that $\cd(A)$ is finite is necessary to get the conclusion in Theorem \ref{thm:cdhgtt}.
\end{remark}

\section{Functors between module categories}\label{sec:tsfmod}

In this section, as a second application of Theorem \ref{thm:raftw}, we discuss twisted seprability in the classical situation  of functors between module categories that are induced by algebra maps, which was the very first example studied in \cite{nvdbvo}. 

Let $\varphi : A \to B$ be an algebra map. Associated to $\varphi$ are the restriction functor $\varphi_* : {_B \mathcal M} \to {_A\mathcal M}$ (already mentioned in Section \ref{sec:tsf}) and the induction functor $\varphi^* : {_A \mathcal M} \to {_B\mathcal M}$, $M \mapsto  B\otimes_AM$. The pair $(\varphi^*,\varphi_*)$ form a pair of adjoint functors whose respective counit and unit are given by
 \begin{align*}
  \eta_V : V &\longmapsto \varphi_*(B\otimes_A V) \quad \quad \varepsilon_M : B\otimes_A \varphi_*(M) \to M \\
v &\longmapsto 1\otimes_A v \quad \quad \quad \quad  \quad \quad\quad \quad b\otimes_A x \longmapsto b\cdot x 
 \end{align*}

\begin{proposition}\label{prop:twalgmor}
 Let $\varphi : A \to B$ be an algebra map.
\begin{enumerate}
 \item 
The following assertions are equivalent.
\begin{enumerate}
 \item The induction functor $\varphi^* : {_A \mathcal M} \to {_B\mathcal M}$ is twisted separable of standard type with the free $A$-module $A$ in the corresponding class $\mathcal F$.
\item  There exist an algebra automorphism $\theta$ of $B$ and a linear map $E : B \to A$ such that $E(1)=1$ and for any $a,a'\in A$ and any $b\in B$, we have
$$E\left(\theta\varphi(a)b\varphi(a')\right) = aE(b)a'$$
\end{enumerate}
\item The following assertions are equivalent.
\begin{enumerate}
 \item The restriction functor $\varphi_* : {_B \mathcal M} \to {_A\mathcal M}$ is twisted separable with twisted separability datum $(\theta_*, \mathcal F, \delta_{-})$ for an algebra automorphism $\theta$ of $A$ and the free $B$-module $B$ belongs to the class $\mathcal F$.
\item There exist an algebra automorphism $\theta$ of $A$, a linear isomorphism $\Gamma : B\to B$, and an element
$\sum_i b_i\otimes_A c_i \in B\otimes_A \theta^{-1}_*\varphi_*(B)$ such that for any $a\in A$ and $b \in B$, we have
$$\Gamma(\varphi(a)b) =\varphi\theta(a) \Gamma(b), \quad \sum_i  bb_i\otimes_A c_i = \sum_i b_i\otimes_A c_ib, \quad \sum_ib_i\Gamma(c_i)=1$$
\end{enumerate}
\end{enumerate}

\end{proposition}

\begin{proof}
(1) Assume that we are given $E : B\to A$ and an automorphism $\theta$ of $B$ with the above properties, and let us construct a twisted separability $(\Theta, \mathcal F, \theta_{-})$ datum of standard type for $\varphi^*$. We take $\Theta =\theta_*$ and $\mathcal F$ the class of free $A$-modules. Then $\varphi^*(P)=B\otimes_AP$ is a free $B$-module if $P$ is a free $A$-module, and we take the isomorphism $\theta_P$ as in \ref{thetaP}.  Since $E(b\varphi(a)) = E(b)a$ for any $\in B$ and $a \in A$, we can define, for any left $A$-module $V$, a linear map
\begin{align*}
 \nu_V : B\otimes_A V &\longrightarrow V \\
b\otimes_Av &\longmapsto E(b).v 
\end{align*}
Using our condition on $E$ and $\theta$, it is immediate that $\nu_V$ defines an $A$-linear map $\varphi_*(\theta_*(B\otimes_A V))\to V$, and clearly $\nu$ defines a natural transformation. Finally the verification that for a free $A$-module $P$ one has $\nu_P\circ \varphi_*(\theta_P) \circ\eta_P={\rm id}_P$ is immediate, so we conclude from Theorem \ref{thm:raftw} that $\varphi^*$ is indeed twisted separable, and is of standard type by construction. 

Conversely assume that $\varphi^*$ is twisted separable of standard type, with corresponding twisted separability datum 
$(\theta_*, \mathcal F, \theta_{-})$ for an automorphism $\theta$ of $B$. Let $$\nu : \varphi_*\theta_*(B\otimes_A \bullet) \to \bullet$$ be the natural transformation provided by Theorem \ref{thm:raftw}. We consider $$\nu_A : \varphi_*\theta_*(B\otimes_AA)\to A$$ and put $E(b)=\nu_A(b\otimes_A 1_A)$. The $A$-linearity of $\nu_A$ gives $E\left(\theta\varphi(a)b\right) = aE(b)$ for any $a \in A$ and $b\in B$, while the naturality of $\nu$, applied to the $A$-linear map $A\to A$, $a' \mapsto a'a$, gives  $E(b)a= \nu_A(b\otimes 1_A)a=\nu_A(b\otimes_A a) = \nu_A(b\varphi(a)\otimes_A1_A)=E(b\varphi(a))$. We also have $E(1)=1$ since $\nu_A\circ\varphi_*(\theta_A)\circ \eta_A={\rm id}_A$, and hence $E$ and $\theta$ satisfy all the required conditions.

(2) Assume that $\varphi_*$ is twisted separable, with twisted separability datum 
$(\theta_*, \mathcal F, \delta_{-})$ for an automorphism $\theta$ of $A$ and with $B$ an object of $\mathcal F$. We therefore have an isomorphism of $A$-modules $\Gamma=\delta_B : \varphi_*(B) \to \theta_*\varphi_*(B)$ whose $A$-linearity gives, for $a\in A$ and $b\in B$, $\Gamma(\varphi(a)b) =\varphi\theta(a) \Gamma(b)$. We also have, by Theorem \ref{thm:raftw}, a natural transformation
$$\xi : \bullet \to B\otimes_A \theta^{-1}_*\varphi_*(\bullet)$$
We put $\xi_B(1) = \sum_ib_i\otimes_A c_i$. The left $B$-linearity of $\xi_B$, together with the naturality of $\xi$, give for $b\in B$, $ \sum_i  bb_i\otimes_A c_i = \xi_B(b) =\sum_i b_i\otimes_A c_ib$. Finally the identity $\sum_ib_i\Gamma(c_i)=1$ follows from the fact that the composition 
$$B \overset{\xi_B}\longrightarrow B\otimes_A \theta_*^{-1} \varphi_*(B) \overset{{\rm id}_B\otimes_A\Gamma}\longrightarrow B\otimes_A\varphi_*(B) \overset{\varepsilon_B}\longrightarrow B$$
is the identity of $B$ (see Theorem 3.1, we can take the identity as adjoint equivalence there since $\theta_*\theta^{-1}_*={\rm 1}= \theta^{-1}_*\theta_*$).

Conversely, assume given $\theta$, $\Gamma$ and $\sum_i b_i\otimes_A c_i \in B\otimes_A \theta^{-1}_*\varphi_*(B)$ as in the statement.  Let us construct a twisted separability $(\Theta, \mathcal F, \theta_{-})$ datum for $\varphi_*$. We take $\Theta = \theta_*$ and $\mathcal F$ the class of free $B$-modules.
We can see $\Gamma$ as an isomorphism $\varphi_*(B) \to \theta_*\varphi_*(B)$, and this extends to an isomorphism 
$\delta_P : \varphi_*(P) \to \theta_*\varphi_*(P)$ for any free $B$-module $P$.

We define $\xi_M : M \to B\otimes_A \theta_*^{-1}\varphi_*(M)$ by $\xi_M(x) = \sum_i b_i\otimes_A c_i\cdot x$ for any $B$-module $M$. It is clear that $\xi_M$ is linear and that this defines a natural transformation $\xi : \bullet \to B\otimes_A \theta^{-1}_*\varphi_*(\bullet)$.
The identity $\sum_ib_i\Gamma(c_i)=1$ implies that the composition
$$B \overset{\xi_B}\longrightarrow B\otimes_A \theta_*^{-1} \varphi_*(B) \overset{{\rm id}_B\otimes_A\Gamma}\longrightarrow B\otimes_A\varphi_*(B) \overset{\varepsilon_B}\longrightarrow B$$
is the identity of $B$, and this is easily seen to hold for any free $B$-module. Hence Theorem \ref{thm:raftw} ensures the twisted separability of $\varphi_*$ with the announced twisted separability datum.
\end{proof}

\begin{remark}
Consider again the situation when $A$ is a Hopf algebra and $R$ is a left $A$-Galois object. 
 Then the uniqueness of adjoint functors implies that the canonical functor $F_R : {_A\mathcal M} \to {_R\mathcal M}_R$ is in fact isomorphic with the induction functor $\kappa^*$, so an approach to study its twisted separability could be to use Proposition \ref{prop:twalgmor}. However, the general answer given by this result looks more obscure to us, and our feeling is that the path we followed in Section \ref{sec:apphg} is the most convenient.
\end{remark}

\end{document}